\documentclass[11pt,draft,reqno]{amsart}

\usepackage{amsmath,amssymb,amscd,amsfonts,amsthm}
\usepackage{dsfont}
\usepackage[all]{xy}
\usepackage{enumerate}
\usepackage{dsfont}
\usepackage{bbold}
\usepackage{mathbbol}
\usepackage{mathrsfs}

\usepackage{array,amscd,latexsym, verbatim}

{\catcode`\@=11
\gdef\n@te#1#2{\leavevmode\vadjust{%
 {\setbox\z@\hbox to\z@{\strut#1}%
  \setbox\z@\hbox{\raise\dp\strutbox\box\z@}\ht\z@=\z@\dp\z@=\z@%
  #2\box\z@}}}
\gdef\leftnote#1{\n@te{\hss#1\quad}{}}
\gdef\rightnote#1{\n@te{\quad\kern-\leftskip#1\hss}{\moveright\hsize}}
\gdef\?{\FN@\qumark}
\gdef\qumark{\ifx\next"\DN@"##1"{\leftnote{\rm##1}}\else
 \DN@{\leftnote{\rm??}}\fi{\rm??}\next@}}
\DeclareFontFamily{OT1}{wncyr}{\hyphenchar\font45 }
\DeclareFontShape{OT1}{wncyr}{m}{n}{%
   <5> <6> <7> <8> <9> gen * wncyr
   <10> <10.95> <12> <14.4> <17.28> <20.74>  <24.88>wncyr10}{}
\DeclareFontShape{OT1}{wncyr}{m}{it}{%
   <5> <6> <7> <8> <9> gen * wncyi
   <10> <10.95> <12> <14.4> <17.28> <20.74> <24.88> wncyi10}{}
\DeclareFontShape{OT1}{wncyr}{m}{sc}{%
   <5> <6> <7> <8> <9> <10> <10.95> <12> <14.4>
   <17.28> <20.74> <24.88>wncysc10}{}
\DeclareFontShape{OT1}{wncyr}{b}{n}{%
   <5> <6> <7> <8> <9> gen * wncyb
   <10> <10.95> <12> <14.4> <17.28> <20.74> <24.88>wncyb10}{}
\input cyracc.def

\DeclareMathSizes{9}{9}{7}{5} 

\theoremstyle{plain}

\newtheorem{theorem}{Theorem}

\newtheorem{lemma}{Lemma}
\newtheorem{remark}{Remark}
\newtheorem{corollary}{Corollary}

\newtheorem*{problemnonumber}{Problem}

\theoremstyle{definition}

\newtheorem{definition}{Definition}

\newtheorem{nothing*}[theorem]{}
\newtheorem{subnothing*}[sub]{}

\theoremstyle{remark}

\def\bA1{{\mathbf A}\!^1}
\def\P1{{\bf P}^1}

\begin{document}

\title[The Jordan property for Lie groups and complex spaces]
{The Jordan property for Lie groups
and\\ automorphism groups of complex spaces}

\author[Vladimir L. Popov]{Vladimir L. Popov${}^*$}
\address{Steklov Mathematical Institute,
Russian Academy of Sciences, Gub\-kina 8, Moscow\\
119991, Russia}

\email{popovvl@mi.ras.ru}

\thanks{${}^*$\,This work is supported by the RSF under a grant 14-50-00005.}

\begin{abstract} We prove that the family of
all connected $n$-dimensional real Lie gro\-ups is uniformly Jordan for every $n$.\;This implies that
all algebraic groups (not necessarily affine) over fields of cha\-racte\-ristic zero and some transformation groups of
complex spaces and Riemannian manifolds are Jordan.
\end{abstract}

\maketitle

\subsection*{1.\;Introduction}
We recall
the
definition introduced in \cite[Def.\;2.1]{P1}:
\begin{definition} \label{def}
Given a group $G$, put $$J_G:=\underset{F}{\rm sup}\;\underset{A}{\rm min}\, [F:A],$$
where $F$ runs over all finite subgroups of $G$ and $A$ runs over all normal abelian subgroups of $F$.\;If
$J_G\neq \infty$, then $G$ is called a {\it Jordan group}
and
$J_G$ is called the {\it Jordan constant} of $G$.
In this case, we also say
 that $G$ enjoys the {\it Jordan property}.
\end{definition}

Informally, the Jordan property of $G$ means that all finite subgroups of $G$ are ``almost abelian'' in the sense that they are extensions of abelian groups by groups taken from only a finite list.\;Definition \ref{def} is inspired by the classical theorem of Jordan \cite{Jo} claim\-ing that $J_{{\rm GL}_n(\ell)}\neq \infty$ holds
for every $n$ and every field $\ell$ of characteristic zero.\;If $\ell$ is algebraically closed, then, for every fixed $n$, the constant $J_{{\rm GL}_n(\ell)}$ is independent of $\ell$, so we denote it simply by $J(n)$.\;It has been computed in
\cite{Co}; in particular,
\begin{equation*}
J(n)=(n+1)!\;\; \mbox{for  all $n\geqslant 71$ and $n=63, 65, 67, 69$.}
\end{equation*}
For more examples of Jordan groups see \cite{P2}.

Below
variety means algebraic variety  over a fixed algebraically closed field $k$ of characteristic zero; in particular, any algebraic group is defined over $k$.\;If $G$ is either an algebraic group or a topological group, $G^0$ denotes the identity component of $G$.

After
being posed  seven years ago in  \cite[Sect.\;2]{P1}
(see also \cite[Sect.\;2]{P2}), the following problem
was explored by
a number of re\-sear\-chers (see the most recent brief survey and references in \cite[Sect.\;1]{PS2}):

\begin{problemnonumber} Describe varieties $X$ for which the group ${\rm Aut}(X)$ is Jordan.
\end{problemnonumber}

At present  (April 2018) it is still unknown
whether there are
varieties $X$ such that
the group ${\rm Aut}(X)$ is non-Jordan
(note that complex manifolds whose
automorphism groups are non-Jordan do exist, see below Remark \ref{rem}). On the other hand, by now for many types of varieties $X$
it is shown that the group ${\rm Aut}(X)$ is Jordan.
In particular,  S.\;Meng and D.-Q.\;Zhang recently proved
the following
\begin{theorem}[{{\rm \cite[Thm.\;1.6]{MZ}}}] \label{proX}  For every projective variety $X$, the group ${\rm Aut}(X)$ is Jordan.
\end{theorem}

Given a variety $X$, we denote by ${\rm Aut}(X)^0$ the identity component of
${\rm Aut}(X)$ in the sense of  \cite{Ra}; see also\;\cite{P3}.\;By \cite[Cor.\;1]{Ra}, if $X$ is complete,
then
${\rm Aut}(X)^0$
is a connected (not necessarily affine) algebraic group.\;Jordan's the\-o\-rem
cited above
implies
the claim
that every affine algebraic group is Jordan;
see \cite[Thm.\;2]{P2}.\;The key ingre\-dient of the   proof of  Theorem \ref{proX} given in   \cite{MZ} is the proof
that the extension of this claim
to all (i.e., not necessarily affine) algebraic groups holds true.\;The
latter
proof is rather involved.

In the present note we obtain, with a very short proof,
a general result, from which
the above-mentioned extension
immediately follows (see Theo\-rem \ref{ag} below).\;Namely, we prove that
every finite-dimensional connected real Lie group is Jordan
(the more precise and
general statements are formulated in Theorems\;\ref{LieJ}, \ref{Jcs}, and Corollary \ref{uJ} below).
Then in Sections 5--7 we apply this to showing
that some transformation groups of complex spaces and Riemannian manifolds are Jordan
(see Theorems \ref{cc}, \ref{CCM}, \ref{hyp}, \ref{conv}, and \ref{iso} below).

The question of whether the Lie groups are Jordan
was posed to me by A.\;M.\;Vershik (see \cite[95:20]{P5}) whom I thank.\;I am grateful to Yu.\;G.\;Zarhin for the valuable comments.

\subsection*{2.\;Lie groups}
We now explore the Jordan property for finite-dimensional real Lie groups $G$.\;Note that non-Jordan groups of this type do exist, because
every discrete group is a $0$-dimen\-sional real Lie group and there are non-Jordan discrete groups (see \cite[1.2.5]{P2}).\;Therefore, the Jordan pro\-perty of $G$ can be expected only under some constraint on the compo\-nent group $G/G^0$.

To formulate this restriction we recall the following definition introduced in
\cite[Def.\;2.9]{P1}:

\begin{definition}
Given a group $H$, put $$b_H^{\ }:=\underset{F}{\rm sup}\;|F|,$$ where $F$ runs over all finite subgroups of $H$.\;If $b_H^{\ }\neq \infty$, then the group $H$ is called {\it bounded}.
\end{definition}

In particular, every finite group $H$ is bounded and $b_H^{\ }=|H|$.

In Theorems \ref{LieJ}, \ref{Jcs} and Corollary \ref{ff} below, we consider the class of finite-dimensional real Lie groups $G$ whose component group $G/G^0$ is bounded. Note that every compact Lie group $K$ belongs to this class, because $K/K^0$ is finite.

\begin{theorem}\label{LieJ}
Let $G$ be a finite-dimensional real Lie group whose
compo\-nent group $G/G^0$ is bounded.\;Then $G$
is Jordan.
\end{theorem}
\begin{proof} By \cite[Lem.\;2.11]{P1} (or \cite[Thm.\;5]{P2}), we may (and shall) assume that
$G$ is connected.\;This assumption implies  the existence of a compact Lie subgroup $K$ of $G$ such that every compact subgroup of $G$ is conju\-gate to that of $K$ (see, e.g., \cite[Chap.\;XV, Thm.\;3.1(iii)]{Ho}).\;In par\-ticular, every finite subgroup of $G$ is conjugate to that of $K$.\;This and Definition \ref{def} show that $G$ is Jordan if and only if $K$ is, and if they are, then
\begin{equation}\label{GK}
J_G=J_K.
\end{equation}

Being compact, the group $K$ admits a faithful finite-dimensional repre\-senta\-tion, i.e., is isomorphic to a subgroup of ${\rm GL}_m(\mathbb R)$ for some $m$ (see, e.g., \cite[Chap.\;5, \S2, Thm.\,10]{OV}).\;Since
the latter group is Jordan, $K$ is Jordan as well
(see \cite[Thm.\;3(i)]{P2}.\;This
completes the proof.
\end{proof}

\begin{corollary}\label{ff} For every finite-dimensional real Lie group $G$ whose compo\-nent group $G/G^0$ is bounded, the set of isomorphism classes of all finite simple subgroups of $G$ is finite.
\end{corollary}

We now dwell on estimating the Jordan constants of Lie groups whose component group is finite, with a view of proving that the class of such groups enjoys a property  stronger than that
of
all its members to be Jordan (see Corollary \ref{uJ} below).\;Seeking only this goal, we did not seek to improve the estimates obtained.

\begin{lemma}\label{ms} Let $S$ be a simply connected simple affine algebraic group.\;Then the minimum ${\rm rdim}\,S$ of dimensions of faithful linear algebraic representati\-ons of $S$ is given by the following table:

\vskip 3mm

\begin{center}
\begin{tabular}{c||c|c|c|c|c|c|c|c|c|c}
{\rm \footnotesize type of} $S$ &$\underset{\ell\geqslant 1}{{\sf A}_\ell}$ &$\underset{\ell\geqslant 2}{{\sf B}_\ell}$ &$\underset{\ell\geqslant 2}{{\sf C}_\ell}$ &$\underset{\ell\geqslant 3, \;\ell \;\mbox{\rm \footnotesize odd}}{{\sf D}_\ell}$ &$\underset{\ell\geqslant 4, \;\ell \;\mbox{\rm \footnotesize even}}{{\sf D}_\ell}$&${\sf E}_6$ &${\sf E}_7$&${\sf E}_8$&${\sf F}_4$&${\sf G}_2$\\
\hline
${\rm rdim}\,S$&$\ell +1$&$2^\ell$&$2\ell$&$2^{\ell -1}$&$2\ell+2^{\ell-1}$&$27$&$56$&$248$&$26$&$7$
\end{tabular}
\end{center}
\end{lemma}

\begin{remark} {\rm  In the proof of Lemma \ref{ms} below,   a faithful representation of $S$ of dimension ${\rm rdim}\,S$ is explicitly specified for each type of $S$.
}
\end{remark}

\begin{proof}[Proof of Lemma {\rm\ref{ms}}] By
Lef\-schetz's prin\-ciple (see, e.g., \cite[VI.6]{Si}), we may (and shall) assume that $k$ is $\mathbb C$.\;We fix a maximal torus $T$ of $S$. Let $\alpha_1,\ldots, \alpha_\ell\in ({\rm Lie}\,T)^*$,  $\varpi_1,\ldots,\varpi_\ell\in ({\rm Lie}\,T)^*$, and $\alpha_1^\vee,\ldots, \alpha_\ell^\vee\in {\rm Lie}\,T$ be respectively the system of simple roots, fundamental weights, and simple coroots of ${\rm Lie}\,T$ with respect to a fixed Borel subalgebra of ${\rm Lie}\,S$  containing ${\rm Lie}\,T$; we number them as in \cite{OV}.

The
center $Z$ of $S$ is a finite subgroup of $T$.\;Fix a subset $\widetilde Z$ of ${\rm Lie}\,T$ whose image under the exponential map ${\rm Lie}\,T\to T$ is the set of all nonidentity elements of $Z$.

For every dominant weight $\lambda\in ({\rm Lie}\,T)^*$,
let $R(\lambda)$ be an irreducible repre\-sentation of  ${\rm Lie}\,S$ with
the highest weight\;$\lambda$.\;The dimension of $R(\varpi_i)$ for every $i$ is specified in
\cite[Ref.\;Chap., \S2, Table 5, pp.\;299--305]{OV}.\;Note that Weyl's dimension formula implies
\begin{equation}\label{Wdf}
\textstyle \dim R\big (\sum_{i=1}^\ell \lambda_i\varpi_i\big)\geqslant
\dim R\big(\sum_{i=1}^\ell \mu_i\varpi_i\big)\quad \mbox{if $\lambda_i\geqslant \mu_i$ for every $i$.}
\end{equation}
Since $S$ is simply connected, $R(\lambda)$
is the differential of a finite-dimensional linear algebraic representation ${\mathcal R}(\lambda)$ of
$S$.\;Since $S$ is simple, for every finite set $D$ of nonzero dominant weights and ${\mathcal R}(D):=\bigoplus_{\lambda\in D}\mathcal R(\lambda)$, we have $\ker {\mathcal R}(D) \subseteq Z$.\;Hence
\begin{equation}\label{cri}
\begin{gathered}
\mbox{${\mathcal R}(D)$ is faithful
$\iff $
for every
$x\in \widetilde Z$
there is
$\lambda\in D$
with $\lambda(x)\notin \mathbb Z$.}
\end{gathered}
\end{equation}

As is well known,
$\dim
\mathcal R(\varpi_1)$ is
the minimum of dimensions of nonzero finite-dimensional algebraic representations of $S$
(see \cite[pp.\,299--305]{OV}).

If $S$ is of type ${\sf E_8}$, ${\sf F_4}$, or ${\sf G_2}$,  then $Z$ is trivial; hence in this case $\mathcal R(\varpi_1)$ is faithful and therefore we have the equality
\begin{equation}\label{pi1}
{\rm rdim}\,S=\dim
R(\varpi_1),
\end{equation}
 which proves the claim of Lemma \ref{ms} for these types.

 If $S$ is of type ${\sf A}_\ell$ or ${\sf  C}_\ell$, then $S$ is respectively ${\rm SL}_{\ell+1}$ and ${\rm Sp}_{2\ell}$.\;Since for these groups
 $\mathcal R(\varpi_1)$ is the tautological faithful representa\-tion,
  in this case \eqref{pi1} holds as well, which proves the claim of
Lemma \ref{ms} for these types.

For the other types, we apply \eqref{cri} to the set $\widetilde Z$ taken from
\cite[Ref.\;Chap., \S2, Table 3, p.\;298]{OV}.\;Below is used that for any $\lambda_i, \mu_i\in k$,
\begin{equation}\label{val}
\mbox{the value of $\textstyle \sum_{i=1}^{\ell} \lambda_i\varpi_i\in ({\rm Lie}\,T)^*$ in $\sum_{i=1}^{\ell}\mu_i\alpha_i^\vee\in {\rm Lie}\,T$ is
$\sum_{i=1}^{\ell} \lambda_i\mu_i$}.
\end{equation}

If $S$ is of type ${\sf E}_7$, then $\widetilde Z$ consists of only one element $\zeta:=(\alpha_1^\vee+\alpha_3^\vee+\alpha_7^\vee)/2$.\;By \eqref{val}, we have
$\varpi_1(\zeta)=1/2\notin \mathbb Z$, so $\mathcal R(\varpi_1)$ is faithful.\;There\-fore, in this case again \eqref{pi1} holds, which proves the claim of Lemma \ref{ms} for this type.

If $S$ is of type ${\sf E}_6$, then $\widetilde Z$ consists of two elements
$\zeta:=(\alpha_1^\vee-\alpha_2^\vee+\alpha_4^\vee-\alpha_5^\vee)/3$ and $2\zeta$.\;Since
$\varpi_1(z)=1/3\notin \mathbb Z$,  $\varpi_1(2z)=2/3\notin \mathbb Z$, in this case again $\mathcal R(\varpi_1)$ is faithful; whence \eqref{pi1} holds.\;This proves the claim of Lemma \ref{ms} for this type.

If  $S$ is of type ${\sf B}_\ell$, then $\widetilde Z$ consists of only one element $\alpha_\ell^\vee/2$.\;This and \eqref{cri}, \eqref{val} imply that $\mathcal R(D)$ is faithful
if and only if $D$ contains $\sum_{i=1}^\ell \lambda_i\varpi_i$
with odd $\lambda_\ell$.\;Using \eqref{Wdf}, from this we infer that $\mathcal R(\varpi_\ell)$ is
the faithful representation of minimal dimension.\;Hence
${\rm rdim}\,S=\dim R(\varpi_\ell)$. This proves the claim of Lemma \ref{ms} for this type.

If $S$ is of type ${\sf D}_\ell$, $\ell\geqslant 3$, $\ell$ odd, then $\widetilde Z$ consists of three elements
\begin{equation}\label{odd}
\zeta:=(\alpha_1^\vee+\alpha_3^\vee+\cdots +\alpha_{\ell - 2}^\vee)/2+
(\alpha_{\ell-1}^\vee-\alpha_\ell^\vee)/4,\;\;2\zeta,\;\;3\zeta.
\end{equation}
From \eqref{cri}, \eqref{val}, \eqref{odd} we infer that $\mathcal R(D)$ is faithful if and only if $D$ contains
$\sum_{i=1}^\ell\lambda_i\varpi_i$ such that $4$ is coprime to either $\lambda_{\ell-1}$ or $\lambda_\ell$.\;This and \eqref{Wdf} show that $\mathcal R(\varpi_\ell)$ is
the faithful representation of minimal dimension.\;Hence
${\rm rdim}\,S=\dim R(\varpi_\ell)$, proving the claim of
Lemma \ref{ms} for this type.

If $S$ is of type ${\sf D}_\ell$, $\ell\geqslant 4$, $\ell$ even, then $\widetilde Z$ consists of three elements
\begin{equation}\label{even}
\zeta_1:=(\alpha_1^\vee+\alpha_3^\vee+\cdots +\alpha_{\ell - 1}^\vee)/2,\;\;
\zeta_2:=(\alpha_{\ell-1}^\vee+\alpha_\ell^\vee)/2,\;\;\zeta_1+\zeta_2.
\end{equation}
Hence if $\mathcal R(D)$ is faithful, then $D$ contains $\sum_{\i=1}^\ell\lambda_i\varpi_i$ with odd $\lambda_{\ell}$ or $\lambda_{\ell-1}$ and  $\sum_{\i=1}^\ell\mu_i\varpi_i$ with odd $\mu_i$ for some odd $i\neq \ell-1$. On the other hand, since in this case $Z$ is not cyclic, Schur's lemma implies that $|D|\geqslant 2$. From this it is not difficult to deduce  that
$\mathcal R(\varpi_1)\oplus \mathcal R(\varpi_\ell)$ is the faithful representation of minimal dimension.\;Hence
${\rm rdim}\,S=\dim R(\varpi_1)+ \dim R(\varpi_\ell)= 2\ell+2^{\ell-1}$. This completes the proof of
Lemma \ref{ms}.
\end{proof}

\begin{corollary}\label{upb}
Every simply connected simple affine algebraic group of rank $\ell$ admits a faithful linear algebraic representation of dimension at most $2^{\ell}+10$.
\end{corollary}

\begin{proof}
Clearly if an algebraic group admits a faithful linear algebraic repre\-sentation, then it admits
 a faithful linear algebraic representation of any bigger dimension.\;In view of this, the claim follows from
the inequality ${\rm rdim}\,S\leqslant 2^{\ell}+10$, which, in turn, follows from Lemma \ref{ms}:\;indeed, the latter shows that ${\rm rdim}\,S\leqslant 2^\ell$ if the type of $S$ differs from ${\sf F}_4$ and ${\sf G}_2$, and that ${\rm rdim}\,S=2^\ell+10$ and $2^\ell+3$ respectively for the types
${\sf F}_4$ and ${\sf G}_2$.
\end{proof}

\begin{theorem}\label{Jcs}
Let $G$ be an $n$-dimensional real Lie group whose
compo\-nent group $G/G^0$ is bounded.\;Then
\begin{equation}\label{est}
J_G\leqslant b_{G/G^0}J\big(n (2^{n}+10)\big)^{b_{G/G^0}}.
\end{equation}
\end{theorem}

\begin{proof} By \cite[Lem.\;2.11]{P1} (or \cite[Thm.\;5]{P2}), we may (and shall) assume that $G$ is connected; in particular,
\begin{equation}\label{b1}
b_{G/G^0}=1.
\end{equation}

We use the notation of the proof of Theorem \ref{LieJ}.\;Since $G$ is connected, $K$ is connected, too; see \cite[Chap.\;XV, Thm.\;3.1(ii)]{Ho}. Hence (see \cite[\S1, Prop.\;4]{Bo}) there are
\begin{enumerate}[\hskip 4.2mm\rm(i)]
\item the compact simply connected simple Lie groups $K_1,\ldots, K_d$;
\item a compact torus $S$;
\item a group epimorphism with finite kernel
\begin{equation}\label{tK}
\pi\colon \widetilde K:=K_1\times\cdots\times K_d\times S\to K.
\end{equation}
\end{enumerate}

By \cite[Thm.\;3(ii)]{P2}, from (iii) we infer that
\begin{equation}\label{cov}
J_K\leqslant J_{\widetilde K}.
\end{equation}

Every $K_i$ is a real form of the corresponding  simply connected simple  complex affine algebraic group.\;The rank $\ell_i$ of the latter is equal to that of $K_i$.\;By Corollary \ref{upb} we then conclude that $K_i$ admits an embedding in ${\rm GL}_{2^{\ell_i}+10}(\mathbb C)$. Since
$\ell_i\leqslant \dim \widetilde K=\dim K\leqslant n$, this in turn implies that $K_i$ admits an embedding in ${\rm GL}_{2^{n}+10}(\mathbb C)$.\;Clearly, $S$ admits an embedding in ${\rm GL}_{\dim S}(\mathbb C)$, and therefore, in view of $\dim S\leqslant \dim \widetilde K$, also in ${\rm GL}_{2^{n}+10}(\mathbb C)$.
This and the definition of $\widetilde K$ (see \eqref{tK}) show that $\widetilde K$ admits an embedding in the direct product of
$d+1$ copies of ${\rm GL}_{2^{n}+10}(\mathbb C)$, hence in ${\rm GL}_{(d+1)(2^{n}+10)}(\mathbb C)$. In turn, since, in view of \eqref{tK}, we have $d+1\leqslant \dim \widetilde K$, from this we infer that $\widetilde K$ admits an embedding in ${\rm GL}_{n(2^{n}+10)}(\mathbb C)$; whence,
\begin{equation}\label{wKJ}
J_{\widetilde K}\leqslant J\big(n(2^{n}+10)\big).
\end{equation}

Now, putting \eqref{GK}, \eqref{cov}, \eqref{wKJ},
\eqref{b1} together, we complete the proof.
\end{proof}

Recall from  \cite{MZ} the following
\begin{definition}
A family $\mathcal F$ of groups is called {\it uniformly Jordan} if every group in $\mathcal F$ is Jordan and there is an integer $J_{\mathcal F}$ such that $J_G\leqslant J_{\mathcal F}$ for every $G\in\mathcal F$.
\end{definition}

\begin{corollary}\label{uJ} Fix an integer $n\geqslant 0$.\;Let $\mathcal L_n$ be the family of all connected $n$-dimensi\-onal
real Lie groups.\;Then
\begin{enumerate}[\hskip 2.2mm\rm(i)]
\item the family
$\mathcal L_n$ is uniformly Jordan;
\item
one can take
$J_{{\mathcal L_n}}=J\big(n(2^{n}+10)\big).$
\end{enumerate}
\end{corollary}
\begin{proof} This follows from \eqref{est} because $b_{G/G^0}=1$ for every $G\in\mathcal L_n$.
\end{proof}

\begin{corollary}\label{ff} For every integer $n\geqslant 0$, the set of isomorphism classes of finite simple groups embeddable in $n$-dimensional connected real Lie groups  is finite.
\end{corollary}

\subsection*{3.\;Algebraic groups}
We now consider several applications of
Theorems \ref{LieJ} and \ref{Jcs}. First, we apply them to algebraic groups, answering
Question 1.2 in \cite{MZ}:

\begin{theorem}\label{ag}
\label{aJ} Every
{\rm(}not necessarily af\-fi\-ne{\rm)} $n$-di\-men\-sional algebraic
group $\;G$ over an algebraically closed field $k$ of cha\-rac\-teristic $0$ is Jordan.\;Moreover,
\begin{equation}\label{JA}
J_G\leqslant [G:G^0]J\big(n(2^{2n+1}+20)\big)^{[G:G^0]}.
\end{equation}
\end{theorem}

\begin{proof} In this case, $G/G^0$ is finite.\;By
Lef\-schetz's prin\-ciple,
we may (and shall) assume that $k$ is $\mathbb C$.\;Then $G$ has a structure of $2n$-dimen\-sional real Lie group whose identity component is $G^0$.\;The claim then follows from Theorem \ref{Jcs}.
\end{proof}

Statement (i) of the next corollary is one of the main results of  \cite{MZ}:

\begin{corollary}\label{appp}
Fix an integer $n\geqslant 0$.\;Let $\mathcal A_n$ be the fa\-mi\-ly of all {\rm(}not neces\-sa\-ri\-ly af\-fi\-ne{\rm)} connected $n$-dimensi\-onal
algebraic groups over an algebraically closed field $k$ of characteristic $0$.\;Then
\begin{enumerate}[\hskip 2.2mm\rm(i)]
\item {{\rm (\cite[Thm.\;1.3]{MZ})}}
the family
$\mathcal A_n$ is uniformly Jordan;
\item one can take
$ J_{\mathcal A_n}=J\big(n(2^{2n+1}+20)\big).$
\end{enumerate}
\end{corollary}
\begin{proof} This follows from \eqref{JA}.
\end{proof}

\subsection*{4.\;Automorphism groups of complex spaces}
The
next application is
to auto\-mor\-phism groups of complex spaces.

Let $C$ be a (not necessarily
reduced) complex space.\;There exists
a topo\-logy on ${\rm Aut}(C)$ with respect to which ${\rm Aut}(C)$  is  a topological group (see \cite[2.1]{Akh}).

\begin{theorem}\label{cc} For every compact complex space $C$, the group
${\rm Aut}(C)^0$ is Jordan.
\end{theorem}

\begin{proof} By \cite{Ka}, the compactness of $C$ implies that  ${\rm Aut}(C)$ is a complex Lie group.\;The claim then follows from Theorem \ref{LieJ}.
\end{proof}

We do not know whether the statement of Theorem \ref{cc} remains
true if  ${\rm Aut}(C)^0$
is replaced by ${\rm Aut}(C)$, i.e., whether the ``complex version'' of  Ghys' con\-jecture holds true.\;By \cite[Thm.\;1.5]{PS2}, the answer is affir\-mative if $C$ is a connected compact two-dimensional complex manifold.\;By
Theorem \ref{proX}, it is also affirmative if  $C$ is a projective variety.\;More generally, it is affirmative if $C$ is a normal compact K\"ahler variety \cite{Ki}.\;On the other hand, we recall that by \cite{Ak} there are connected smooth compact real manifolds whose diffeomorphism groups are non-Jordan (this disproves the original
Ghys' con\-jecture).

\begin{remark}\label{rem} {\rm
There are connected noncompact complex manifolds, whose automorphism groups are non-Jordan.\;Indeed,
by \cite{Wi}, for any
count\-able group $\Gamma$, there is a noncompact Riemann surface $M$ such that ${\rm Aut}(M)$ is isomorphic to $\Gamma$;
whence the claim because of the existence of
countable non-Jordan groups (see \cite[Sect.\;1.2.5]{P2}).}
\end{remark}

In actual fact, using the idea exploited earlier in \cite{P4}, one can prove more than
is said
in Remark \ref{rem}, showing
the existence of connected comp\-lex mani\-folds with monstrous automorphism groups, namely:

\begin{theorem} \label{5} There is a $3$-dimensional simply connected noncompact com\-p\-lex manifold $M$\;such that
\begin{enumerate}[\hskip 2.2mm\rm(i)]
\item the group ${\rm Aut}(M)$ contains an isomorphic copy of every finitely pre\-sentable {\rm(}in particular, every finite{\rm)} group;
\item every such copy is a discrete transformation group of $M$ acting freely.
\end{enumerate}
\end{theorem}

\begin{proof} It follows (see, e.g.,\;\cite[Thm.\;12.29]{R2}) from
Higman's embed\-ding theorem
\cite{Hi}
that there is a universal finitely pre\-sen\-ted group, i.e., a finitely presented group $\mathcal U$ containing as a subgroup an isomorphic copy of every finitely presented group.\;In turn, by \cite[Cor.\;1.66]{ABCKT} the finite presentability of $\mathcal U$ implies the existence of a connected $3$-dimensional
compact complex manifold $B$ whose fundamental group is iso\-mor\-phic to $\mathcal U$.\;Consider the universal cover $\pi\colon \widetilde B\to B$. Then $\widetilde B$ is a simply connected
noncom\-pact $3$-dimensional complex manifold and the deck transformation group of $\pi$ is a subgroup of ${\rm Aut}\,\widetilde B$ isomorphic to $\mathcal U$, which acts on
$\widetilde B$ freely. Hence
one can take $M=\widetilde B$.
\end{proof}

\begin{remark}
{\rm
For $M$ from Theorem \ref{5}, the group ${\rm Aut}(M)$ is non-Jordan, because for every integer $n$, there is a finite simple group of order $>n$ (cf.\;\cite[Example\;4]{P2}.}
\end{remark}

\begin{theorem} \label{CCM}
Fix an integer $n\geqslant 0$. Let $\mathcal C_n$ be the family of groups
${\rm Aut}(M)^0$, where $M$ runs over all connected compact complex manifolds of complex dimension $n$. Then
\begin{enumerate}[\hskip 2.2mm\rm(i)]
\item the family ${\mathcal C}_n$ is uniformly Jordan;
\item  one can take $J_{\mathcal C_n}=
J\big(\!(2n^2+n)(2^{2n^2+n}+10)\big)$.
\end{enumerate}
\end{theorem}

\begin{proof} For $G:={\rm Aut}(M)^0$, let $K$ be as
in the proof of Theorem \ref{LieJ}.\;According to Mont\-gomery--Zip\-pin's theorem \cite[Chap.\;VI, Sect.\;6.3.1, Thm.\;2]{MoZi},
$\dim K\leqslant 2n^2+n$. Since, clearly, $J(m)$ is a nondecreasing function of $m$, the latter inequality, \eqref{GK}, and Theorem \ref{Jcs} yield
$J_G\leqslant J\big(\!(2n^2+n)(2^{2n^2+n}+10)\big)$. This proves (i) and (ii).
\end{proof}

\subsection*{5.\;Automorphism groups of hyperbolic complex manifolds}
The next
appli\-cation is  to
comp\-lex manifolds hyperbolic in the sense of Kobayashi (in particular, to boun\-ded domains in $\mathbb C^n$).

\begin{theorem}\label{hyp} Fix an integer $n\geqslant 0$.\;Let $\mathcal H_n$ be the family of
groups ${\rm Aut}(M)^0$, where $M$ runs over all connected complex manifolds hyperbolic in the sense of Kobayashi and of complex dimension $n$. Then
\begin{enumerate}[\hskip 2.2mm\rm(i)]
\item
the family
$\mathcal H_n$ is uniformly Jordan;
\item one can take
$ J_{\mathcal H_n}=J\big((2n+n^2)(2^{2n+n^2}+10)\big)$;
\item for every point $x\in M$, the ${\rm Aut}(M)$-stablizer ${\rm Aut}(M)_x$ of $x$ is Jordan and
$
J_{{\rm Aut}(M)_x}\leqslant J(n).
$
\end{enumerate}
\end{theorem}

\begin{proof} Let $M$ be a connected complex manifolds hyperbolic in the sense of Kobayashi and of complex dimension $n$.\;By \cite[Thms.\;2.1, 2.6]{Ko2}, ${\rm Aut}(M)$ is a real Lie group of dimension $\leqslant 2n+n^2$; whence (i) and (ii) by Theorems \ref{LieJ} and \ref{Jcs}.\;By \cite[Thm.\;2.6]{Ko2}, the isotropy representation of ${\rm Aut}(M)_x$ is faithful and its image is isomorphic to a subgroup of the unitary group ${\rm U}(n)$; whence (iii).
\end{proof}
\begin{remark}
{\rm
 The group ${\rm Aut}(M)^0$ in the formulation of Theorem \ref{hyp} cannot be replaced by ${\rm Aut}(M)$.\;Indeed, it follows from the construction in \cite{Wi}
that  the Riemann surface $M$ in Remark \ref{rem} is hyperbolic in the sense of Kobayashi.
Therefore there are connected hyperbolic complex manifolds $M$ such that the group ${\rm Aut}(M)$ is not Jordan.
}
\end{remark}

However, as the next theorem shows, for comp\-lex hyperbolic manifolds $M$ of a special type, the Jordan property holds
for the whole ${\rm Aut}(M)$ rather than only for ${\rm Aut}(M)^0$.

\begin{theorem}\label{conv}
For every strongly pseudoconvex bounded domain $M$
with smooth boundary
in $\mathbb C^n$,
the group ${\rm Aut}(M)$ of all biholomorphic transfor\-ma\-tions of $M$ is Jordan.
\end{theorem}

\begin{proof} If the Lie group ${\rm Aut}(M)$ is compact, then the claim follows from Theorem \ref{LieJ}.\;If the group ${\rm Aut}(M)$ is non-compact, then, by the Rosey--Wong theorem \cite{Ros}, \cite{Wo}, the domain $M$ is biholomorphic to the unit ball $B_n$ in $\mathbb C^n$.\;Since ${\rm Aut}(B_n)$ is ${\rm PU}(n, 1)$ (see \cite[Sect.\;2.7, Prop.\;3]{Akh}), and the latter Lie group is  connected (see \cite[Chap.\;IX, Lem.\;4.4]{He}),
the claim then follows from Theorem \ref{LieJ}.
\end{proof}

\begin{corollary}
For every strongly pseudoconvex bounded domain $M$ with smooth boundary in $\mathbb C^n$, the set of isomorphism classes of all finite simple groups of biholomorphic transformations of $M$ is finite.
\end{corollary}

\subsection*{6.\;Isometry groups of Riemannian manifolds}
The last application is to
isometry groups ${\rm Iso}(M)$
of Riemannian mani\-folds $M$.\;They are
topological groups with respect to the compact-open topology
 \cite{Ko1}.

\begin{theorem}\label{iso}
Fix an integer $n\geqslant 0$.\;Let $\mathcal R_n$ be the family of
groups ${\rm Iso}(M)^0$, where $M$ runs over all connected $n$-dimensional Riemannian manifolds. Then
\begin{enumerate}[\hskip 4.2mm\rm(i)]
\item the family ${\mathcal R}_n$ is uniformly Jordan;
\item one can take $J_{\mathcal R_n}=J\big(\!(n^2+n)(2^{(n^2+n-2)/2}+5)\big)$;
\item for every point $x\in M$, the
${\rm Iso}(M)$-stabilizer ${\rm Iso}(M)_x$ of $x$
is Jordan;
\item if the manifold $M$ is compact, then the group ${\rm Iso}(M)$ is Jordan.
\end{enumerate}
\end{theorem}

\begin{proof} It is known
(see, e.g., \cite[Chap.\;II, Thms.\;1.2 and 3.1]{Ko1}) that
${\rm Iso}(M)$
is  a real Lie group of dimension at most $n(n+1)/2$, the group ${\rm Iso}(M)_x$ is compact for every $x$, and the group ${\rm Iso}(M)$ is compact if
the manifold $M$ is compact.\;The claims then follows from combining these facts with Theorems \ref{LieJ} and \ref{Jcs}.
\end{proof}

\begin{remark}{\rm
The group ${\rm Aut}(M)^0$ in the formulation of Theorem \ref{iso} cannot be replaced by ${\rm Aut}(M)$.\;Indeed,
it follows from the construction in \cite{Wi}
that  the Riemann surface $M$ in Remark \ref{rem} is a two-dimensional Riemannian manifold
and ${\rm Aut}(M)={\rm Iso}(M)$. Therefore there are connected Riemannian manifolds $M$ such that the group ${\rm Iso}(M)$ is not Jordan.
}
\end{remark}

\subsection*{7.\;Concluding remarks} In view of \eqref{GK}, computing the Jordan con\-stants of connected real Lie groups is reduced to that of compact such groups. For instance,  the results of  \cite{Co} may be interpre\-ted as computing the Jordan constants of all unitary groups:
$$
J_{{\rm U}_n}=J(n) \;\;\mbox{for every $n$}.
$$
This leads to the following natural

\begin{problemnonumber} Compute the Jordan constants of all simple compact connected real Lie groups.
\end{problemnonumber}

\end{document}